\documentclass[twoside,reqno,11pt]{article}

 \usepackage{hyperref}
 \usepackage{amssymb,amsmath}
 \usepackage[utf8]{inputenc}
 \usepackage{amsthm}
 \usepackage[english]{babel}  %

 \newtheorem{theo}{Theorem}[section]
 \newtheorem{lem}{Lemma}[section]
 \newtheorem{prop}{Proposition}[section]
 \newtheorem{cor}{Corollary}[section]
 \newtheorem{defi}{Definition}[section]
 \newtheorem{rem}{Remark}[section]
 \newtheorem{ex}{Example}[section]
 \newtheorem{notation}{Notation}[section]

\begin{document}

{Daniel Cao Labora $^1$, Rosana Rodr\'iguez-L\'opez $^2$}\\

\textbf{From fractional order equations to integer order equations}\\

This document is a preprint of a currently submitted article. There might be significant differences between this manuscript and the future final version  (adding, deleting or improving some contents). The authors do strongly recommend to the reader to take this document as a sketch (with proofs) of the most important results of a future final version. In this sense, the authors request the reader to consult the final document, when published in a scientific journal, for a more detailed and corrected version.
 \begin{abstract}

The main goal of this article is to show a new method to solve some Fractional Order Integral Equations (FOIE), more precisely the ones which are linear, have constant coefficients and all the integration orders involved are rational. The method essentially turns a FOIE into an Ordinary Integral Equation (OIE) by applying a suitable fractional integral operator.

After discussing the state of the art, we present the idea of our construction in a particular case (Abel integral equation). After that, we propose our method in a general case, showing that it does work
 when dealing with a family of ``additive'' operators over a vector space. Later, we show that our construction is always possible when dealing with any FOIE under the 
 above-mentioned hypotheses.
 Furthermore, it is shown that our construction is ``optimal'' in the sense that the OIE that we obtain has the least possible order.

 \end{abstract}

 \vspace*{-16pt}

\section{Introduction}
The topic of fractional calculus has achieved a relevant position in the mathematical study of problems corresponding to different fields of application, and hence, 
in the last years, its interest has increased considerably \cite{MMKA}. Different concepts of integrals and derivatives of fractional order have been proposed based on the iteration of classical  operators  and the extension of the definitions to non integer orders \cite{KST,Pod,samko}. In many of these concepts, the Gamma function plays an essential role in their construction, and it is also present in the expression of the explicit solutions to the corresponding linear fractional differential equations, and in the study of the properties of the solutions to nonlinear problems. These calculations are sometimes tedious due to the essential differences of fractional operators with respect to the 
basic properties of the classical ones.

In this work, we show a new method to solve linear fractional order integral equations with constant coefficients, in the special case where the orders of  integration that appear are rational.  The method proposed allows to convert each of these problems into a classical integral equation, what is achieved by applying a fractional operator chosen specifically. We illustrate the applicability of the method to different particular examples such as, for instance, Abel equation. We also propose a general theoretical framework, 
useful under the restriction of additivity of operators, and with nice implications for the case of rational order integral equations.

We begin by remembering some basic notions in fractional calculus.
Later, in Section 2, we give some motivations for the problem studied. After that, in Section 3, we establish the general theoretical framework and, in Section 4, the general procedure is exposed based on the concept of generalized polynomials. Besides, the optimality of the construction is justified. Section 5 is devoted to the application of the method to fractional order integral equations with constant coefficients, the study of the properties of their solutions and the presentation of some examples. Finally, in Section 6, some considerations for fractional differential equations are made.

\subsection{Basic tools in fractional calculus}

Throughout this paper, $f$ is assumed to be a real valued function of real variable defined on an interval $(a,b)$. Furthermore, we will assume that $f \in L^1(a,b)$ and $x$ will denote an arbitrary point $x \in (a,b)$. 

The first step is, obviously, to introduce the definitions of fractional integral and fractional derivative.
\begin{defi} 
We define the left Riemann-Liouville fractional integral of $f$ of order $\alpha>0$ from $a \in \mathbb{R} \cup \{-\infty,\infty\}$ as $$(I_{a^+}^{\alpha} f)(x):=\frac{1}{\Gamma(\alpha)} \int_{a}^x \frac{f(t)}{(x-t)^{1-\alpha}}dt \in L^1(a,b),$$ and the right analogue from $b \in \mathbb{R} \cup \{-\infty,\infty\}$ as $$(I_{b^-}^{\alpha} f)(x):=\frac{1}{\Gamma(\alpha)} \int_{x}^b \frac{f(t)}{(t-x)^{1-\alpha}}dt \in L^1(a,b).$$
\end{defi}
\begin{rem} The definition could have been also made for complex values of $\alpha$. However, at the end of the day, we will be only interested in the case where $\alpha$ takes rational values, so it is enough to consider real values of $\alpha$.
\end{rem}
\begin{rem} \label{obs}
The property of $f$ belonging to $L^1(a,b)$ is preserved by fractional integration.
\end{rem}

\begin{defi} We define the left Riemann-Liouville fractional derivative of $f$ of order $0<\alpha<1$ from $a$ as the following value (provided it exists): $$(D_{a^+}^{\alpha} f)(x):=\frac{d}{dx}(I_{a^+}^{1-\alpha}f)(x)=\frac{1}{\Gamma(1-\alpha)} \frac{d}{dx} \int_{a}^x \frac{f(t)}{(x-t)^{\alpha}}dt.$$ The right analogue from $b$ is $$(D_{b^-}^{\alpha} f)(x):=-\frac{d}{dx}(I_{b^-}^{1-\alpha}f)(x)=\frac{-1}{\Gamma(1-\alpha)} \frac{d}{dx} \int_{x}^b \frac{f(t)}{(x-t)^{\alpha}}dt.$$
\end{defi}

It is impossible to extend the previous definition to $\alpha=1$ in a reasonable way since $\Gamma(1-\alpha)=\infty$. Furthermore, for instance, when $f$ is a constant function the integral involved in the previous definition is divergent. So, it is clear that another definition is required for the derivatives of higher orders.

\begin{defi} We define the left Riemann-Liouville fractional derivative of $f$ of order $\alpha>0$, $\alpha \not \in \mathbb{Z}$, from $a$ as the following value (provided it exists): $$(D_{a^+}^{\alpha} f)(x):=\Big(\frac{d}{dx}\Big)^{[\alpha]}(D_{a^+}^{\{\alpha\}}f)(x)=\Big(\frac{d}{dx}\Big)^{[\alpha]+1}(I_{a^+}^{1-\{\alpha\}}f)(x),$$ where $\{\alpha\}$ is the decimal part of $\alpha$ and $[\alpha]$ is the largest integer smaller than or equal to $\alpha$ (integer part of $\alpha$). Furthermore, we give the analogous definition for the right derivative from $b$ as: $$(D_{b^-}^{\alpha} f)(x):=\Big(-\frac{d}{dx}\Big)^{[\alpha]}(D_{b^-}^{\{\alpha\}}f)(x)=\Big(-\frac{d}{dx}\Big)^{[\alpha]+1}(I_{b^-}^{1-\{\alpha\}}f)(x).$$
\end{defi}

If $\alpha$ is a natural number, the left derivative of order $\alpha$ is defined as the usual one. In the right case, it is  the usual one up to a sign $(-1)^\alpha$. 

The essential properties that are going to be used are the following ones, whose proof can be found in \cite{samko}.

\begin{lem}[\cite{samko}] \label{cancel} Given any function $f:\mathbb{R} \to \mathbb{R}$ with $f \in L^1(a,b)$, if $\alpha > \beta> 0$, we have $D_{a^+}^\beta I_{a^+}^\alpha f = I_{a^+}^{\alpha-\beta} f$, and the right analogue is also valid.
\end{lem}
\begin{lem}[\cite{samko}] \label{integral} Given any function $f:\mathbb{R} \to \mathbb{R}$ with $f \in L^1(a,b)$, we have $I_{a^+}^\alpha I_{a^+}^\beta f = I_{a^+}^{\alpha+\beta} f$, and the right analogue is also valid.
\end{lem}
\begin{lem}[\cite{samko}] \label{tecnico} Given any function $f:\mathbb{R} \to \mathbb{R}$ with $f \in I^{\alpha+\beta}(L^1(a,b))$, we have $D_{a^+}^\alpha D_{a^+}^\beta f = D_{a^+}^{\alpha+\beta} f$, and the right analogue is also valid.
\end{lem}

\subsection{The space $\mathcal{S}_{a^+}$}

Until now, we have mentioned that a suitable space to deal with fractional integration is $L^1(a,b)$. Essentially, a fractional integral is an endomorphism in that space and we have the additivity of orders of fractional integrals. However, in the case of fractional derivatives, we have no such a good context. So we will construct a suitable space to work and, from now on, we will use some additional notation. The first step is calling $$\mathcal{S}_{a^+}=I^{\infty}(L^1(a,b)):= \bigcap_{\alpha \in \mathbb{R}^+ \cup \{0\}} I^{\alpha} (L^1(a,b)).$$ 

\begin{rem} The first thing to observe is that the previous intersection is nonempty. The proof of the last sentence is trivial because $f=0$ belongs to the intersection. There are more functions that belong to the intersection, however, there exist many important functions that do not belong to the intersection, for example, $f(x)=x$. 
\end{rem}

Secondly, it would be nice if we could simplify the definition of $\mathcal{S}_{a^+}$ to a countable intersection. This is possible because we have the property $I^{\alpha}(L^1(a,b)) \subset I^{\beta}(L^1(a,b))$ whenever $\alpha > \beta \geq 0$. That is obvious from the fact that, if $f \in I^{\alpha}(L^1(a,b))$, then $f \in I^{\beta}I^{\alpha-\beta}(L^1(a,b)) \subset I^{\beta}(L^1(a,b))$, using $\alpha-\beta>0$ and Lemma \ref{integral}. 

\begin{rem} The identity $$\mathcal{S}_{a^+}=\bigcap_{\alpha \in E} I^{\alpha}_{a^+} (L^1(a,b))$$ holds for any set $E \subset \mathbb{R}$ with no upper bound. So, when checking that a function lies on $\mathcal{S}_{a^+}$, it is enough to see that it is the integral of order $\alpha$ of a function in $L^1(a,b)$ for ``large values'' of $\alpha$.
\end{rem}

\begin{rem} \label{derivarJ} If $f \in \mathcal{S}_{a^+}$, due to the construction of $\mathcal{S}_{a^+}$, we are always under the hypothesis of Lemmas \ref{cancel} and \ref{tecnico}. So, we can ensure that $D_{a^+}^{\alpha}$ is an endomorphism in $\mathcal{S}_{a^+}$ and that the additivity $D_{a^+}^{\alpha} \circ D_{a^+}^{\beta}=D_{a^+}^{\alpha+\beta}$ holds in $\mathcal{S}_{a^+}$ for every $\alpha,\beta \in \mathbb{R}^+ \cup \{0\}$. It is also an evident fact that fractional differentiation is bijective (and so, an isomorphism), since fractional integration is the inverse isomorphism in $\mathcal{S}_{a^+}$.
\end{rem}

\begin{rem} In particular, the functions in $\mathcal{S}_{a^+}$ admit a derivative of any natural order, so we have that $\mathcal{S}_{a^+}$ is a vector subspace of $\mathcal{C}^{\infty}(a,b)$. 

Furthermore, if $f \in \mathcal{S}_{a^+}$, it is trivial that $f(a)$ is well defined and $f(a)=0$ (in the case that $a=\pm \infty$, the evaluation should be thought as a limit). Since all the derivatives of $f$ lie in $\mathcal{S}_{a^+}$, it follows immediately that $f(a)=f'(a)=f''(a)=\dots=0$. The converse also holds, since it is well-known that, if $f(a)=f'(a)=\dots=f^{n-1)}(a)=0$ and $f$ is $n$ times differentiable, we can write $f=I_{a^+}^n f^{n)}$. Hence, if $f$ is infinitely differentiable and all its derivatives vanish at $a$, we can conclude that $f \in \mathcal{S}_{a^+}$. 

In the case that $a \in \mathbb{R}$, this implies that there are no more analytical functions in $\mathcal{S}_{a^+}$ than $f=0$. On the other hand, if $a$ is not finite this does not happen necessarily. For instance, if $f(x)=e^x$, we have $f \in \mathcal{S}_{-\infty}$.
\end{rem}

\section{Motivation}

Abel equation is one of the most studied problems that involves the fractional calculus. There exists a clear physical interpretation of it, and it can be stated in many diverse contexts. This interpretation is quite clear in \cite{shantanu} and the solution to the equation is perfectly explained in \cite{samko}. For us, Abel equation is going to be written in the form
\begin{align} \label{abel}
I_{a^+}^{\frac{1}{2}} x(t)=f(t),
\end{align}
where $f \in L^1(a,b)$ and we are looking for a solution $x \in L^1(a,b)$. The classical procedure to obtain the solution is to apply the operator $D_{a^+}^{1/2}$ to both sides of the equation. Thanks to  Lemma \ref{cancel}, there is no much more to say here because
\begin{align*}
x(t)=D_{a^+}^{\frac{1}{2}}f(t).
\end{align*}
However, we are going to show two additional methods to solve (\ref{abel}). The principal advantage of these methods is that they can be applied to more general equations. In this case, both of them use the identity $I_{a^+}^{\frac{1}{2}} \circ I_{a^+}^{\frac{1}{2}}=I_{a^+}^{{1}}$. 

The first method just deduces, from (\ref{abel}), the equation
\begin{align} \label{rabel}
I_{a^+}^{1} x(t)=I_{a^+}^{\frac{1}{2}} f(t).
\end{align}
Under reasonable hypotheses, linear fractional integral equations with constant coefficients will be proved to have at most one solution. So, it is just about computing the solution of (\ref{rabel}), which has natural integration orders because the right term is known, and checking it in (\ref{abel}), where a fractional integral was present.

The second method cares about the equation 
\begin{align*}
I_{a^+}^{1} y(t)=f(t).
\end{align*}
Clearly, if it does exist a solution $y$ to this problem, then the term $I_{a^+}^{\frac{1}{2}} y(t)$  satisfies $$I_{a^+}^{\frac{1}{2}} \Big(I_{a^+}^{\frac{1}{2}}y(t)\Big)= f(t).$$ So, in that case, $I_{a^+}^{\frac{1}{2}} y(t)$ is  the unique solution to (\ref{abel}).
\section{A suitable framework}

We will summarize our problem in an easy way. As initial objects, we are just given a set with some properties (a structure) and a way to act over the set (an operator). Specifically, assume that $V$ is a vector space over $\mathbb{C}$ and that we have a family of linear operators that will be denoted by $\mathcal{J}=\{J^{\alpha} \in \textnormal{End}(V): \alpha \in \mathbb{Q}^+ \cup \{0\}\}$. So $\mathcal{J}$ is no more than a family of operators where each operator is tagged, in a biunivocal way, as $J^{\alpha}$ for some $\alpha \in \mathbb{Q} \cup \{0\}$.

Evidently, we are not interested in any random election of $\mathcal{J}$. We want our previous labelling to be logical in some sense, so we give the following definition.

\begin{defi} \label{aditivo}
The family $\mathcal{J}$ is said to be additive if $J^{\alpha} \circ J^{\beta} = J^{\alpha + \beta}$ whenever $J^{\alpha}, J^{\beta} \in \mathcal{J}$.
\end{defi}

From now on, $\mathcal{J}$ will be assumed to be additive. The important point to be remembered is that positive rational exponents are allowed in the elements of $\mathcal{J}$. When we think about $\mathcal{J}$ as a family of operators, we are saying, essentially, that the family is formed by natural powers and roots of the operator $J:=J^1$ chosen in a coherent way.

Only with this structure many questions can be asked and, probably, one of the most natural ones is about the solutions $v \in V$ of 
\begin{align} \label{general}
(c_1 J^{\alpha_1} + \cdots + c_l J^{\alpha_l})v=w, \quad \quad 0 \leq \alpha_l < \alpha_{l-1} < \dots < \alpha_1,
\end{align}
where $w \in V$ and $c_i \in \mathbb{C}$, for all $i \in \{1,\dots,l\}$, are known. The previous equation will be called linear fractional order equation (FOE) with constant coefficients and it is evident that (\ref{abel}) and (\ref{rabel}) are particular cases of (\ref{general}). Obviously, solving this kind of equations with such a general formulation is impossible, unless more properties on the structure are required. However, we can make a non-trivial assertion about equations in the form (\ref{general}). Roughly speaking, equations like (\ref{general}) are not more complicated than the similar equations with natural exponents over $J$ (that we will call NOEs). This statement will be specially interesting in the case of linear FOIEs (fractional order integral equations) with constant coefficients, because it will mean the possibility of reducing the computation of the solutions of our FOIE to the solutions of a linear OIE with constant coefficients. In the analogous case of fractional order differential equations, a theoretical solution is possible too, but its practical applications are more limited as we will show later.

\section{The general procedure in terms of generalized polynomials}

As we said before, we want to reduce in a quick and easy way equations of the type (\ref{general}) to similar ones with natural orders. Since the $\alpha_i$, where $i \in \{1,\dots,l\}$, are all rational, we can express them with a common denominator. So, we call this common denominator $q$, and we just rewrite the expression (\ref{general}) as 
\begin{align}\label{2}
c_1J^{\frac{a_1}{q}} v+ \dots + c_l J^{\frac{a_l}{q}} v =w.
\end{align}
Now we want to find an operator in $\langle \mathcal{J} \rangle$ that converts the left side of (\ref{2}) into an analogous expression, but with natural exponents. The first question is, obviously, if this procedure is always possible. The second one should be, if it is possible, how can we construct the operator explicitly.

We will deal with both questions from an algebraic and elegant point of view by introducing the definition of generalized polynomial, which is an analogue to the ordinary case but with fractional degrees. It will be a trivial remark that the existence of the linear operator for any equation like (\ref{2}) is equivalent to the statement ``for any generalized polynomial $p$, it does always exist another generalized polynomial $\widehat{p}$ that makes $p \cdot \widehat{p}$ a polynomial''. Finally, the truthfulness of the assertion concerning generalized polynomials will be proved, ensuring that the procedure shown for solving equations of the form (\ref{2})  does always work.

\subsection{Some notions about polynomials}

\begin{defi} A generalized polynomial (in one variable) is an algebraic expression like 
\begin{align*}
p\,(X)=c_1 X^{\alpha_1} + \dots + c_l X^{\alpha_l},\textnormal{ where } \alpha_1> \dots > \alpha_l \geq 0
\end{align*}and for every $i \in \{1,\dots,l\}$ we have $c_i \in \mathbb{C}$ and $\alpha_i \in \mathbb{Q}^+ \cup \{0\}$. Of course, taking  $q$ as the l.c.m. of the denominators, the previous expression can be rewritten as 
\begin{align*}
p\,(X)=c_1 X^{\frac{a_1}{q}} + \dots + c_l X^{\frac{a_l}{q}},\textnormal{ where } \alpha_1> \dots > \alpha_l \geq 0.
\end{align*}
The set of generalized polynomials will be denoted as $\mathcal{G}$.
\end{defi}

\begin{rem} By defining the sum and product of generalized polynomials in the usual way, it can be proved that $(\mathcal{G},+,\cdot)$ is a $\mathbb{C}$-algebra.
\end{rem}

\begin{rem} It is easy to check that $\mathcal{G}$ is isomorphic, as $\mathbb{C}$-algebra, to $\mathcal{J}$.
\end{rem}

That isomorphism between $\mathbb{C}$-algebras allows us to reformulate our question in terms of generalized polynomials. Given any generalized polynomial $p$, does it always exist another generalized polynomial $\widehat{p}$ such that $p \cdot \widehat{p}$ is a polynomial?

For example, for $p\,(X)=c_1 X^2+c_2 X^{\frac{3}{2}} + c_3$, the choice $$\widehat{p}\,(X)=c_1 X^2-c_2 X^{\frac{3}{2}} + c_3$$ ensures that $(p \cdot \widehat{p}) (X)=c_1^2 X^4 - c_2^2 X^3+ 2c_3 c_1 X^2+c_3^2$ is a polynomial. The natural question that we stated before is if this construction is always possible. To solve this question, an algorithm to construct $\widehat{p}$, which depends only on finding the roots of $p$, will be given.

Our idea rests on the following well-known remark.

\begin{rem} \label{ciclo} Assume that $q \in \mathbb{Z}^+$ and $a \in \mathbb{C}\setminus \{0\}$, then the following decomposition holds $$Y^q-a^q=\prod_{j=0}^{q-1} (Y-a \xi^j)=(Y-a)(Y-a \xi)\cdots (Y-a \xi^{q-1}),$$ where $\xi$ is a primitive $q$-root of $1$. This decomposition follows immediately from the fact that both polynomials of degree $q$ have exactly the same roots and the same principal coefficient. When $a=0$, the remark holds in a trivial way.
\end{rem}

Next, we state and prove our main result.

\begin{theo} For every generalized polynomial $p \in \mathcal{G}$, it does exist another generalized polynomial $\widehat{p} \in \mathcal{G}$ such that $p \cdot \widehat{p} \in \mathbb{C}[X]$.
\end{theo}
\begin{proof}
We begin with $$p\,(X)=c_1 X^{\frac{a_1}{q}} + \dots + c_l X^{\frac{a_l}{q}}=c_1 \Big(X^{\frac{1}{q}}\Big)^{a_1} + \dots + c_l \Big(X^{\frac{1}{q}}\Big)^{a_l}.$$ The Fundamental Theorem of Algebra allows a decomposition like $$p\,(X)=c_1 \Big(X^{\frac{1}{q}}-r_1\Big) \cdots \Big(X^{\frac{1}{q}}-r_n\Big),$$ where we rename the degree as $n=a_1$. If we use the notation
\begin{align*}
p_i\,(X)=\Big(X^{\frac{1}{q}}-r_i\Big), \quad i \in \{1,2,\dots,n\},
\end{align*}
it is clear that a possible choice for $\widehat{p}$ is $\widehat{p}=\widehat{p_1} \cdot \widehat{p_2} \cdots \widehat{p_n}$. So, in fact, it is enough to compute a possibility for every $\widehat{p_i}$. 

The Remark \ref{ciclo} shows that, when choosing $$\widehat{p_i}\,(X)=(X^{\frac{1}{q}}-r_i \xi )(X^{\frac{1}{q}}- r_i \xi^2)\cdots(X^{\frac{1}{q}}-r_i \xi^{q-1}),$$ we have actually that $$(p_i\cdot \widehat{p_i})(X)=\Big({X^{\frac{1}{q}}}\Big)^q-r_i^q=X-r_i^q.$$ So, as we claimed before, we have that the choice $\widehat{p}=\widehat{p_1}\cdots\widehat{p_n}$ gives $$(p\cdot\widehat{p})(X)=c_1(X-r_1^q)(X-r_2^q)\cdots(X-r_n^q),$$ which is a polynomial in $X$ of degree $n$.
\end{proof}

\begin{cor}\label{rrr} If (\ref{general}) is rewritten as 
\begin{align}\label{te}
T v=w,\textnormal{ where } T \in \langle\mathcal{J}\rangle,
\end{align}
then it is possible to find $\widehat{T} \in \langle\mathcal{J}\rangle$ that produces 
\begin{align*}
\widehat{T} \circ T = T \circ \widehat{T} \in \langle\{1,J,J^2,J^3,\dots\}\rangle.
\end{align*}
\end{cor}

From this corollary, two methods that reduce (\ref{general}) to a similar equation with natural orders are developed. The first method gives a set of possible solutions that contains the set of authentic solutions, the second one gives a set of solutions that is contained in the set of all the solutions.

\medskip

\paragraph{\bf Checking method:} We can use Corollary \ref{rrr} to deduce 
\begin{align}\label{tegorro}
(\widehat{T} \circ T) v=(T \circ \widehat{T}) v=\widehat{T} w,\textnormal{ where } T,\widehat{T} \in \langle\mathcal{J}\rangle
\end{align}
from (\ref{te}), so every solution to (\ref{te}) is a solution to (\ref{tegorro}). If we know how to solve equations in natural orders, we know how to solve (\ref{tegorro}). So, if we check the solutions of (\ref{tegorro}) in (\ref{te}), we will find all the solutions of (\ref{te}).

\medskip

\paragraph{\bf Computing method:} This method is a bit more tricky. We deal with the equation
\begin{align}\label{tegorro2}
(\widehat{T} \circ T) v=(T \circ \widehat{T}) v=w,\textnormal{ where } T,\widehat{T} \in \langle\mathcal{J}\rangle.
\end{align}
If we know how to solve equations with integer orders, we know how to solve (\ref{tegorro2}). If $u \in V$ is a solution to (\ref{tegorro2}) it is trivial that $\widehat{T}(u)$ solves (\ref{te}). The fundamental problem of this method is that it does not ensure that every solution is obtained. The advantage is that in some cases, as for instance when dealing with fractional integrals, the right side of the equation can be more treatable in (\ref{tegorro2}) than in (\ref{tegorro}). In the case of fractional integrals, (\ref{te}) will be proved to have at most one solution so, if we find it with the ``computing method'', the procedure will be over.

\subsection{Minimality of the construction}

It has been proved the existence of the linear operator that turns a linear FOE with constant coefficients into a similar one of natural orders. If we think about the ``checking method'', every solution of the fractional equation is a solution of the integral one, but the converse is not necessarily true. So, looking for the solutions of the FOE is just about checking the solutions of the NOE in the expression of the FOE. However, with the idea of minimizing the amount of checks, it would be interesting to guarantee what is the least possible value for the degree of the equation $n$.

If we go back to the previous theorem, we can see that our construction seems pretty reasonable and it would not be a surprise if it is the minimal one, in the previous sense. In fact, we will show that it is the minimal construction when it does not exist a pair of distinct roots $r_i, r_j \in \{r_1,...,r_n\}$ and a $q$-root of unity $\xi$ such that $r_i=\xi r_j$. When this happens, it is obvious that there is no need of working separately with $r_i$ and $r_j$, since $p_j$ is one of the factors on $\widehat{p_i}$ and vice versa.

We will continue using the notation $$p\,(X)=c_1 X^{\frac{a_1}{q}} + \dots + c_l X^{\frac{a_l}{q}},$$ where $q>0$ and $a_1 > a_2 > \dots > a_l \geq 0$.

\begin{lem} The common denominator of the exponents in $\widehat{p}$ is the same as the one in $p$.
\end{lem}

\begin{proof} The proof is trivial, since the other case is an absurd. Without lost of generality, imagine that there are some summands in $\widehat{p}$ whose exponents are not integer multiples of $\frac{1}{q}$. Among these summands, the one with the biggest degree is chosen, let's call it $cX^r$. It is clear that in $(p \cdot \widehat{p})(X)$ there is going to be a summand $c c_1 X^{r+\frac{a_1}{q}}$ that can not be corrected with any other, contradicting that $p \cdot \widehat{p}$ is a polynomial.
\end{proof}

So, because of the previous lemma we can use, from now on, the notation $$\widehat{p}\,(X)=d_1 X^{\frac{a'_1}{q}} + \dots + d_m X^{\frac{a'_m}{q}},$$ where $q>0$ and $a'_1 > a'_2 > \dots > a'_m$.

With the purpose of enlightening the notation, the substitution $Y=X^{\frac{1}{q}}$ will be made. It is well-known that $\mathbb{C}[Y]$ is a principal ideal domain (PID) and, so, a factorization unique domain (FUD). It is obvious that $\mathbb{C}[Y^t]$ is also a FUD, for every $t \in \mathbb{Z}^+$, just because of the change of variable $Z=Y^t$. 

\begin{defi} Given a polynomial $f \in \mathbb{C}[Y]$ and a complex number $y_0 \in \mathbb{C}$, we denote by $\textnormal{ord}_{f}(y_0)$ the maximum natural number such that $(Y-y_0)^{\textnormal{ord}_{f}(y_0)}$ divides $f$.
\end{defi}
\begin{defi} Given a polynomial $f \in \mathbb{C}[Y]$, we denote the set of roots of $f$ as $$R_f=\{y_0 \in \mathbb{C}:f(y_0)=0\}$$ and the set of the $q$-powers of the roots as $$R_f^q=\{y_0^q \in \mathbb{C}:f(y_0)=0\}.$$
\end{defi}
\begin{rem} It should be noted that $\#R_f^q \leq \#R_f$.
\end{rem}

\begin{lem} If $y_0$ is a root of $p\,(X^q)$ and $\xi$ is a primitive $q$-root of unity, then for any choice of $\widehat{p}$ we have that $y_0 \xi^j$ is a root of $p \cdot \widehat{p}$ for every $j \in \{0,\dots,q-1\}$. Furthermore, for any valid choice of $\widehat{p}$ we will always have that $$\displaystyle \prod_{y_0 \in R_f^q} (Y-y_0)^{m_{y_0}}$$ divides $(p \cdot \widehat{p})(X^q)$, where $\displaystyle m_{y_0}=\max_{0 \leq j \leq q-1}\{ord_f(y_0 \xi^j)\}$.
\end{lem}
\begin{proof} We have $$(p \cdot \widehat{p})(X^q)=b_1 X^{s} + \dots + b_s X +b_{s+1}=b_1 Y^{s\cdot q} + \dots + b_s Y^{q} +b_{s+1}.$$ As we mentioned before, there is a unique factorization $$(p \cdot \widehat{p})(X^q)=b_1(Y^q-\beta_1)\cdots(Y^q-\beta_s).$$ So, if $y_0$ is a root of $p\,(X^q)$ with multiplicity $k$, $Y-y_0$ divides at least $k$ different factors $Y^q-\beta_{i}$ and $y_0 \xi^j$ is going to be necessarily a root of multiplicity $k$ of $(p \cdot \widehat{p})(X^q)$ for every $j \in \{0,1,\dots,q-1\}$. Obviously, changing the role of $y_0$ for each $y_0\xi^j$ implies that the minimum possible amount of factors is $m_{y_0}$. In particular, when $Y-y_0$ appears in $p\,(X^q)$ with multiplicity $k$, the same factor appears in $\widehat{p}\,(X^q)$ with multiplicity $m_{y_0}-k$.
\end{proof}

\begin{ex}
Now, we illustrate how does the procedure work, already using $X^{\frac{1}{q}}$ instead of $Y$. For example, consider the generalized polynomial $$p\,(X)=(X^{\frac{1}{4}}+2)^2(X^{\frac{1}{4}}-2)(X^{\frac{1}{4}}+3).$$ Our optimal proposal is $$\widehat{p}\,(X)=(X^{\frac{1}{4}}-2)(X^{\frac{1}{4}}+2i)^2(X^{\frac{1}{4}}-2i)^2(X^{\frac{1}{4}}-3)(X^{\frac{1}{4}}+3i)(X^{\frac{1}{4}}-3i),$$ that makes $$(p \cdot \widehat{p})(X)=\prod_{j=1}^4 \Big((X^{\frac{1}{4}}-2 i^j)^2 (X^{\frac{1}{4}}-3 i^j)\Big)=(X-16)^2(X-81).$$
\end{ex}

\section{Fractional order integral equations}

In this section, we will be concerned about linear FOIEs with constant coefficients of the type
\begin{align}\label{1}
c_1I_{a^+}^{r_1} x(t)+ \dots + c_nI_{a^+}^{r_n} x(t) =f(t),
\end{align} 
where $a,b \in \mathbb{R}$, $f \in L^1(a,b)$, $r_i \in \mathbb{Q}^+ \cup \{0\}$ and $c_i \in \mathbb{C}$ for every $i \in \{1,\dots,n\}.$ The general solution to this equation, even when $r_i \in \mathbb{R}$, is known and can be found in \cite[page 846]{samko}. However, the procedure used to find the solution is quite technical and relies on some results of \cite{vladimirov}, which is written in Russian.

The difficulty for solving (\ref{1}) has motivated some authors to find a technique that allows an easier solution, even if it is only valid in a less general case. One of the most known simplifications to (\ref{1}) is to assume that the $r_i$ are non-negative rational numbers. That case has been treated, for example, in \cite{shantanu}. So, from now on, we will also work with the hypothesis $r_i \in \mathbb{Q}^+ \cup \{0\}$ for all $i \in \{1,\dots,n\}$.

We can return to the previous syntax, that appears in (\ref{general}), by taking $J^{\alpha}=I_{a^+}^{\alpha}$, $V=L^1(a,b)$, $v=x(t)$ and $w=f(t)$. The property exposed on Lemma \ref{integral} ensures that we are under the hypothesis of additivity, given in Definition \ref{aditivo}.

\subsection{Uniqueness of solution to linear integral equations with constant coefficients}

It is widely known that the solution to a linear ODE (of order $n$) with constant coefficients is an affine space of dimension $n$. In the case of an integral equation, we have no such result. We are going to include a proof of some lemmas that supply that result. What is going to be proved is that any homogeneous linear OIE with constant coefficients has only the trivial solution.

\begin{lem} If $f \in \mathcal{C}^{\infty}(a,b)$, we have that if the problem 
\begin{align} \label{integralentera}
c_0 I_{a^+}^{n} x(t)+ \dots + c_{n-1} I_{a^+}^{1} x(t) + c_{n} x(t) =f(t), \textnormal{ where } c_0 \neq 0,
\end{align}
has solutions in $L^1(a,b)$, then the solutions are in $\mathcal{C}^{s}(a,b)$ and they are even analytical if $f$ is analytical too.
\end{lem}

\begin{proof} If $s \in \{0,1,\dots n\}$ is the largest possible number such that $c_s \neq 0$, from (\ref{integralentera}), we deduce that 
\begin{align*}
\sum_{j=0}^{s-1} c_j I_{a^+}^{n-j} x(t)-f(t)=-c_s I_{a^+}^{n-s} x(t).
\end{align*}
If we assume that $f \in \mathcal{C}^{\infty}(a,b)$ and $x \not \in \mathcal{C}^{\infty}(a,b)$, we can find the smallest natural number $r$ such that $f \not \in C^r(a,b)$ (we can have $r=0$ if $x$ is not continuous, but anycase we know that lies in $L^1(a,b)$). However, because of the construction of $r$, the left member is clearly a sum of functions that lie in $C^r(a,b)$, giving a contradiction.

If $f$ is analytical, we have $f \in \mathcal{C}^{\infty}(a,b)$ and it has been already shown that $x \in \mathcal{C}^{\infty}(a,b)$. So, it is possible to differentiate $n$ times at both sides in (\ref{integralentera}) leading to the expression 
\begin{align} \label{deducidaD}
c_0 x(t)+ \dots + c_{n-1} D^{n-1} x(t) + c_{n} D^n x(t) =D^n f(t).
\end{align}
The last expression is known to have an $n$ dimensional space of analytical solutions so, since all the solutions to (\ref{integralentera}) are solutions to (\ref{deducidaD}), all the solutions to (\ref{integralentera}) will be analytical.
\end{proof}

\begin{lem} The problem (\ref{integralentera}) with $f=0$ and $a \in \mathbb{R}$ has only the trivial solution.
\end{lem}
\begin{proof} 
Because of the previous development we already know that $x$ is  a solution to
\begin{align*}
c_0 x(t)+ \dots + c_{n-1} D^{n-1} x(t) + c_{n} D^n x(t) =0.
\end{align*}
Hence $x$ lies, a priori, in a vector space of dimension $s$, where $s$ is the maximum value $s \in \{0,1,\dots,n\}$ such that $c_s \neq 0$. By differentiating the equation $n-s$ times and substituting $t=a$, we get the condition $x(a)=0$. From that information, differentiating the equation $n-s+1$ times and substituting $t=a$, we deduce  the condition $x'(a)=0$. This procedure can be done until differentiating $n-1$ times and the conditions obtained are $x(a)=x'(a)=\dots=x^{s-1)}(a)=0$, that in fact show that the unique possible solution is $x=0$.
\end{proof}
\begin{ex} The procedure is not valid when $a$ is not finite. For instance, we have that the equation $$D^1 x(t)=x(t)$$ has, as solutions, $x(t)=k e^t$, with $k \in \mathbb{R}$. It is obvious that, if $a=-\infty$, all the solutions to the equation fullfil the condition $x(a)=0$. Furthermore, the integral equation $$I_{-\infty}^1 x(t)=x(t)$$ has more than one solution. In fact, $x(t)=k e^t$ solves the last integral equation for every $k \in \mathbb{R}$.
\end{ex}
\begin{cor} If $f \in L^1(a,b)$, then the problem (\ref{integralentera}) with $a \in \mathbb{R}$ has at most one solution.
\end{cor}
\begin{proof} The proof is trivial, since the existence of two different solutions would imply that their difference is a non trivial solution to (\ref{integralentera}) with $f=0$, contradicting the previous lemma.
\end{proof}
This result extends in an obvious way to a FOIE like (\ref{1}), because the ``checking method'' ensures that the set of solutions  to (\ref{1})  in $L^1(a,b)$ is  contained in the set of solutions to some equation of the type (\ref{integralentera}), that is, at most, an element.
\begin{cor} Any linear rational order integral equation with constant coefficients of the type (\ref{1}), where $a \in \mathbb{R}$, has at most one solution in $L^1(a,b)$.
\end{cor}
\begin{ex} Sometimes the set of solutions is empty as we can see with the equation $I_0^{1}x(t)=1$, which will lead to incoherences when $t=0$. 
\end{ex}
\subsection{An example}
\begin{ex}
We show, as an example, how to solve the equation $$T x(t)=I_{0}^1 x(t) + 5 I_{0}^{3/4} x(t) + 2 I_{0}^{1/2}  x(t)  -20 I_{0}^{1/4} x(t) -24 x(t) = e^t.$$
The associated generalized polynomial is $$p\,(X)=(X^{\frac{1}{4}}+2)^2(X^{\frac{1}{4}}-2)(X^{\frac{1}{4}}+3).$$ Our optimal proposal is $$\widehat{p}\,(X)=(X^{\frac{1}{4}}-2)(X^{\frac{1}{4}}+2i)^2(X^{\frac{1}{4}}-2i)^2(X^{\frac{1}{4}}-3)(X^{\frac{1}{4}}+3i)(X^{\frac{1}{4}}-3i),$$ which is associated to the integral operator 
\begin{equation}\label{tebarra}
\begin{split}
\widehat{T}=& I_{0}^2 - 5 I_{0}^{7/4} + 23 I_{0}^{3/2} - 85 I_{0}^{5/4} + 190 I_{0}^{1} - 440 I_{0}^{3/4}\\
&  + 672 I_{0}^{1/2} - 720 I_{0}^{1/4} + 864 \textnormal{ id}.
\end{split}
\end{equation}
That produces $$(p \cdot \widehat{p})(X)=(X-16)^2(X-81)=X^3- 113 X^2 + 2848 X - 20736.$$

The computing method deals with the equation $$(T \circ \widehat{T}) y(t)=I_{0^+}^3 y(t)- 113 I_{0^+}^2 y(t) + 2848 I_{0^+}y(t) - 20736 y(t)=e^t.$$ We have already shown that the solutions to the previous equation are analytical. So, if we differentiate the equation three times, we can arrive to the equation 
\begin{align}\label{d}
y(t) - 113 y'(t) + 2848 y''(t) - 20736 y'''(t) =e^t.
\end{align}
The general solution to (\ref{d}) can be checked to be of the form $$y(t) = c_1 e^\frac{t}{81} + c_2 e^\frac{t}{16} + c_3 e^\frac{t}{16} t - \frac{e^t}{18000},$$ but only one of those functions is indeed a solution to (\ref{d}). From (\ref{d}) we can derive the following system of equations
\begin{align*}
- 20736 y(0)=1,\\
- 20736 y'(0) + 2848 y(0)=1,\\
- 20736 y''(0)+ 2848 y'(0)- 113 y(0)=1,
\end{align*}
which has as unique solution
\begin{align*}
y(0) = \frac{-1}{20736}, \quad y'(0) = \frac{-737}{13436928}, \quad y''(0) = \frac{-1932835}{34828517376}.
\end{align*}
Now, we can use this information to compute the values of $c_1, c_2$ and $c_3$ from
\begin{align*}
y(0)=c_1 + c_2 - \frac{1}{18000},\\
y'(0)=\frac{c_1}{81} + \frac{c_2}{16} + {c_3} - \frac{1}{18000},\\
y''(0)=\frac{c_1}{81^2}  + \frac{c_2}{16^2} + \frac{c_3}{8} - \frac{1}{18000}.
\end{align*}
Hence, we obtain $$c_1 = \frac{1}{27378000}, \quad c_2 = \frac{71}{9734400}, \quad c_3 = \frac{1}{3993600}.$$
Therefore, $$y(t) = \frac{1}{27378000}e^\frac{t}{81}+ \Big(\frac{71}{9734400} + \frac{1}{3993600} t\Big) e^\frac{t}{16}- \frac{1}{18000}e^t.$$
Finally, we obtain that
\begin{align*}
x(t)=\widehat{T} y(t)=\widehat{T} \Big(\frac{1}{27378000}e^\frac{t}{81}+ \Big(\frac{71}{9734400} + \frac{1}{3993600} t\Big) e^\frac{t}{16}- \frac{1}{18000}e^t\Big),
\end{align*}
where $\widehat{T}$ is given by (\ref{tebarra}), so we can see that the solution is a sum of exponentials and Mittag-Leffler functions.
\end{ex}

\section{The case of fractional order differential equations}

One would expect an analogue result for fractional order differential equations. Theoretically, we can do a similar construction. The set $\mathcal{S}_{a^+}$ was constructed precisely to guarantee the additivity of orders of fractional derivatives and that any fractional differential operator is an endomorphism in $\mathcal{S}_{a^+}$. However, assuming that the solutions to  linear FODEs with constant coefficients lie in $\mathcal{S}_{a^+}$ is not a very good hypothesis. For example, if we consider $$D_{a^+}^{\frac{1}{2}} x(t) - x(t)=0,$$ it is well known that there are non-trivial solutions of the problem such as $x(t)=(t-a)^{\frac{-1}{2}}$. However, those solutions are not detected by our procedure because they do not lie in $\mathcal{S}_{a^+}$. If they did, we would deduce, by additivity, $$(D_{a^+}^{\frac{1}{2}}+I)(D_{a^+}^{\frac{1}{2}}-I) x(t)=D_{a^+}^{1} x(t) - x(t)=0.$$ So, in fact, the only possibility for a solution in $\mathcal{S}_{a^+}$ is $x(t)=k e^t$, with $k \in \mathbb{R}$. Nevertheless, in that case, the original equation does not hold except for the trivial value $k=0$.

On the other hand, consider the equation $$D_{a^+}^{\frac{3}{2}} x(t) -  D_{a^+}^{1}x(t) - D_{a^+}^{\frac{1}{2}}x(t)+x(t)=0.$$ In this case, it is an easy check that $x(t)=k e^t$ is a solution. This can be trivially shown because the left side equals $$(D_{a^+}^{\frac{1}{2}}-D_{a^+}^{0})(D_{a^+}^{1} - D_{a^+}^{0})x(t)=0,$$ where we have used that $x(t)=k e^t$ lies in $\mathcal{S}_{a^+}$ and that $x(t)=k e^t$ is annihilated by the differential operator $D_{a^+}^{1} - D_{a^+}^{0}$. 

\subsection{Solutions in $\mathcal{S}_{a^+}$}

So, what can we say about differential equations? Of course, we have already shown that we can not expect an astonishing result as in the integral case. But, we have also exemplified that in some cases we are able to extract some solutions. If we look only for the solutions that lie in $\mathcal{S}_{a^+}$, the general procedure is valid.

\begin{notation}
The set $A_{\lambda}$ is defined as $$A_{\lambda}=\{p(t) e^{\lambda t}: p \textnormal{ is a polynomial}\}, \textnormal{ where } \lambda \in \mathbb{C}.$$
We also define $$A = \bigoplus_{\lambda \in \mathbb{C}} A_{\lambda}.$$
\end{notation}

\begin{rem}
We should note that the set $A$ contains the set of solutions of any homogeneous linear ODE with constant coefficients.
\end{rem}
\begin{prop}
The following assertions are valid:
\begin{itemize}
\item If $a = -\infty$, then $\displaystyle \mathcal{S}_{a^+} \cap A = \bigoplus_{\Re (\lambda) >0} A_{\lambda}$.
\item If $a = \infty$, then $\displaystyle \mathcal{S}_{a^+} \cap A = \bigoplus_{\Re (\lambda) <0} A_{\lambda}$.
\item If $a \in \mathbb{R}$, then $\mathcal{S}_{a^+} \cap A = \{0\}.$
\end{itemize}
\end{prop}

\begin{proof} It is trivial that $$\frac{d}{dx} p(x)e^{\lambda x} = (\lambda p(x) + p'(x))e^{\lambda x}.$$ If $\Re (\lambda )> 0$, we can deduce that $$\int_{-\infty}^x (\lambda p(t) + p'(t))e^{\lambda t}dt= p(x)e^{\lambda x}.$$ Analogously, if $\Re (\lambda) < 0$, we have $$\int_{\infty}^x (\lambda p(t) + p'(t))e^{\lambda t}dt= p(x)e^{\lambda x}.$$  This shows the inclusion $\supset$ for the first two cases. Furthermore, the inclusion $\subset$ is obvious in both cases, since the functions of $S_{-\infty}$ (resp., $S_{\infty}$) necessarily vanish at $-\infty$ (resp., $\infty$) and this is impossible for any function of $A_{\lambda}$ when $\Re(\lambda) \geq 0$ (resp., $\Re(\lambda) \leq 0$).

If $a \in \mathbb{R}$ and $f \in \mathcal{S}_{a^+} \cap A$, we already know that $f$ is analytical, since any function in $A$ is analytical. Furthermore, the fact $f \in\mathcal{S}_{a^+}$ ensures that $f(a)=f'(a)=f''(a)=\dots=0$. The analytic character of $f$ forces it to be identically zero.
\end{proof}

\begin{rem} The previous result indicates us, esentially, a thing with double taste. For any linear FODE with constant coefficients, we  have the uniqueness of solution in $\mathcal{S}_{a^+}$ when $a \in \mathbb{R}$. For example, when the FODE is homogeneous, the unique solution is the trivial one. However, this property is not very nice when dealing with an ODE because we are used to have a vector space of solutions of certain dimension. Furthermore, it is also obvious that, if we want a linear FODE with constant coefficients to have a solution, we will need the right term to be in $\mathcal{S}_{a^+}$, which is quite restrictive. Nevertheless, it is not very complicated to show that $\mathcal{S}_{a^+}$ is dense in the space of continuous functions that vanish at $a$, so future work could be developed in this direction.
\end{rem}


 \bigskip \smallskip

 \it

 \noindent

$^1$ Departamento de Estad\'istica, An\'alisis Matem\'atico y Optimizaci\'on\\
Facultad de Matem\'aticas \\
Campus Vida \\
Santiago de Compostela -- 15782, SPAIN  \\[4pt]
e-mail:  daniel.cao@usc.es

$^2$ Departamento de Estad\'istica, An\'alisis Matem\'atico y Optimizaci\'on\\
Facultad de Matem\'aticas \\
Campus Vida \\
Santiago de Compostela -- 15782, SPAIN  \\[4pt]


\begin{thebibliography}{99}
 \normalsize
 
 \bibitem{shantanu} U. Ghosh, S. Sarkar, S. Das, Solution of system of linear fractional differential equations with modified derivative of Jumarie type.  \emph{American Journal of Mathematical Analysis} \textbf{3}, No 3 (2015),  72--84; DOI: 10.12691/ajma-3-3-3.
 
\bibitem{KST} A.A. Kilbas, H.M. Srivastava, J.J. Trujillo,  \emph{Theory and Applications of Fractional Differential Equations}. In: North-Holland Mathematics Studies, vol. 204 (2006). 

\bibitem{Pod} I. Podlubny,  \emph{Fractional Differential Equations}. Academic Press, San Diego (1999).
 
\bibitem{samko}S.G. Samko, A.A. Kilbas, O.I. Marichev,  \emph{Fractional Integrals and Derivatives: Theory and Applications}. Gordon and Breach Science Publishers,  Yverdon (1993).

\bibitem{MMKA}J.A. Tenreiro Machado, F. Mainardi, V. Kiryakova, T. Atanackovi\'{c}, Fractional Calculus: D'o\`u venons-nous? Que sommes-nous? O\`u Allons-nous?
\emph{Fractional Calculus \& Applied Analysis} \textbf{19}, No 5 (2016),  1074--1104; DOI: 10.1515/fca-2016-0059.

\bibitem{vladimirov} V.S. Vladimirov,  \emph{Generalized Functions in Mathematical Physics}. Nauka, Moscow (1979).









\end{thebibliography}
\end{document}